\documentclass[final,times]{elsarticle}
\usepackage{commath}
\usepackage[english]{babel}
\usepackage{amsmath,amsthm}
\usepackage{amsfonts}
\usepackage{mathscinet}
\usepackage{textcomp}
\usepackage{units}
\usepackage{enumerate}
\usepackage{srcltx}
\usepackage{graphics}
\usepackage{color}
\usepackage{mathrsfs}
\usepackage{amsthm}
\usepackage{rotating}
\usepackage{graphicx}
\usepackage[bookmarks=true, pdfstartview={FitH}]{hyperref}
\usepackage[margin=1in]{geometry}

\usepackage[arrow,matrix,curve]{xy}
\newdir{ >}{{}*!/-1em/\dir{>}}

\newcommand{\al}{\alpha}
\newcommand{\bt}{\beta}
\newcommand{\dl}{\delta}
\newcommand{\Dl}{\Delta}
\newcommand{\lb}{\lambda}
\newcommand{\Lb}{\Lambda}
\newcommand{\vf}{\varphi}
\newcommand{\sg}{\sigma}

\newcommand{\bZ}{\mathbb{Z}}

\newcommand{\pa}{\partial}
\newcommand{\lra}{\longrightarrow}
\newcommand{\wt}{\widetilde}

\newcommand{\ol}{\overline}
\newcommand{\os}{\overset}
\newcommand{\sbs}{\subset}
\newcommand{\vnth}{\varnothing}

\newcommand{\rlm}{\underset{\longrightarrow}{\lim}\,}
\newcommand{\llm}{\underset{\longleftarrow}{\lim}\,}

\newcommand{\Hom}{\operatorname{Hom}}
\newcommand{\Ext}{\operatorname{Ext}}
\newcommand{\Ker}{\operatorname{Ker}}
\newcommand{\Coker}{\operatorname{Coker}}
\newcommand{\IIm}{\operatorname{Im}}

\theoremstyle{dgthm}
\newtheorem{theorem}{Theorem}
\newtheorem{corollary}{Corollary}
\newtheorem{lemma}{Lemma}

\theoremstyle{dgdef}
\newtheorem{definition}{Definition}

\newtheorem{remark}{Remark}

\begin{document}


\title{The Tautness Property of Homology Theory}

\author{Anzor Beridze and Leonard Mdzinarishvili }

\address{Department of Mathematics,
Faculty of Exact Sciences and Education,
Batumi Shota Rustaveli State University,
35, Ninoshvili St., Batumi,
Georgia; \\
School of Mathematics, Kutaisi International University, Youth Avenue, 5th Lane, Kutaisi 4600, Georgia\\
e-mail:~a.beridze@bsu.edu.ge~~~anzor.beridze@kiu.edu.ge
}

\address{Department of Mathematics,
Faculty of Informatics and Control Systems,
Georgian Technical University,
77, Kostava St., Tbilisi,
Georgia; e-mail:~l.mdzinarishvili@gtu.ge}


\begin{abstract} {\color{black} The tautness for a cohomology theory is formulated and studied by various authors. However, the analogous property is not considered for a homology theory. In this paper, we will define and study this very property for the Massey homology theory. Moreover, we will prove that the Kolmogoroff and {\color{black}the} Massey homologies are isomorphic on the category of locally compact, paracompact spaces and proper maps. Therefore, we will obtain the same result for the Kolmogoroff homology theory.}
\end{abstract}	
	


\maketitle
\paragraph*{\bf Introduction}
{\color{black} Let $A$ be a closed} subspace of a topological space $X$ and $\{U\}$ {\color{black}be a} system of neighborhoods $U$ of $A$, directed by inclusion. Then for each cohomology theory $h^*$ there is a natural homomorphism 
$$
i^*: \rlm h^*(U) \lra h^*(A). \eqno(*)
$$
It is said that $A$ is tautly embedded in {\color{black}the} space $X$ if {\color{black}the} homomorphism $i^*$ is an isomorphism \cite[\S 6.1]{11}. The Alexander-Spanier cohomology  on the category of paracompact Hausdorff spaces and continuous maps \cite[Theorem {\color{black}2 \S 6.6}]{11}  and the Massey cohomology $H_c^*$ on the category of  locally compact Hausdorff spaces and proper maps  \cite[Theorem 6.4, \S 6.4]{9} are the examples of cohomolog{\color{black}ies}  for which {\color{black}any closed} subspace $A$ is tautly embedded in $X$ . {\color{black}  It is natural to ask whether or not the analogous property fulfills for an exact homology theory as well. Therefore, the aim is to investigate a natural homomorphism  
$$
i_*:h_*(A) \lra \llm h_*(U) \eqno(**)
$$
for a homology theory.}
In this paper, it is proved that for the Massey homology $\os{{\,}_M}{H}_*$, there exists an infinite exact sequence on the category of locally compact Hausdorff spaces $X$, which includes the homomophism $i_*$. In particular, we have the following main theorem.\\

{\color{black}
{\bf Theorem 2.}
	The system $\{N\}$ of closed neighborhoods $N$ of closed subspace $A$ of a locally compact Hausdorff space $X$, directed by inclusion, induces the following exact sequence
	\begin{gather*}
	\cdots \lra \llm^{(2k+1)} \os{{\,}_M}{H}_{n+k+1}(N) \lra \cdots \lra \llm^{(3)} \os{{\,}_M}{H}_{n+2}(N) \lra \llm^{(1)} \os{{\,}_M}{H}_{n+1}(N) \lra \\
	\lra \os{{\,}_M}{H}_{n}(A,G) \os{i_n}{\lra} \llm \os{{\,}_M}{H}_{n}(N) \lra  \llm^{(2)} \os{{\,}_M}{H}_{n+1}(N) \lra \cdots \lra \llm^{(2k)} \os{{\,}_M}{H}_{n+k}(N) \lra \cdots\,,
	\end{gather*}
	where $\os{{\,}_M}{H}_*(N) =\os{{\,}_M}{H}_*(N,G)$ is the  Massey  homology \cite[\S 4.6]{9} of closed neighborhood $N$ with coefficient in an abelian group $G$.}

{\color{black} It is natural to study the same property for other exact homology theories \cite{7}, \cite{12}, \cite{3}, \cite{10}. Consequently, in the second part of the paper, it is proved that the Kolmogoroff \cite{7,2} and {\color{black}the} Massey \citep{9} homologies are isomorphic on the category of locally compact, paracompact spaces and proper maps. Using the obtained result, we will show that for the Kolmogoroff \cite{7}, the Milnor \cite{10}  and the Steenrod \cite{12} homology theories the following properties hold:
	
{\bf Corollary 5.}
	{\rm a)} If $X$ is a locally compact, paracompact Hausdorff space, then for the system $\{N\}$ of closed neighborhoods $N$ of a closed subspace $A$ of  $X$, there is an infinite exact sequence
	\begin{gather*}
	\cdots \lra \llm^{(2k+1)} \os{{\,}_K}{H}_{n+k+1}(N) \lra \cdots \lra \llm^{(3)} \os{{\,}_K}{H}_{n+2}(N) \lra \llm^{(1)} \os{{\,}_K}{H}_{n+1}(N) 
	\lra \os{{\,}_K}{H}_{n}(A,G) \nonumber \lra \\
	\os{i_n}{\lra} \llm\os{{\,}_K}{H}_{n}(N) \lra \llm^{(2)} \os{{\,}_K}{H}_{n+1}(N) \lra \cdots \lra \llm^{(2k)} \os{{\,}_K}{H}_{n+k}(N) \lra \cdots\,,
	\end{gather*}
	where $\os{{\,}_K}{H}_*(N)=\os{{\,}_K}{H}_*(N,G)$ is the Kolmogoroff homology.
	
	{\rm b)} If  $X$ is a compact Hausdorff space, then for the system $\{N\}$ of closed neighborhoods $N$ of a closed subspace $A$ of $X$, there is an infinite exact sequence
	\begin{gather*}
	\cdots \lra \llm^{(2k+1)} \os{{\,}_{Mi}}{H}_{n+k+1}(N) \lra \cdots \lra \llm^{(1)} \os{{\,}_{Mi}}{H}_{n+1}(N) \lra \os{{\,}_{Mi}}{H}_{n}(A) \lra\\
	\os{i_n}{\lra} \llm\os{{\,}_{Mi}}{H}_{n}(N) \lra \llm^{(2)} \os{{\,}_{Mi}}{H}_{n+1}(N) \lra \cdots \lra \llm^{(2k)} \os{{\,}_{Mi}}{H}_{n+k}(N) \lra \cdots\,,
	\end{gather*}
	where $\os{{\,}_{Mi}}{H}_*(N)=\os{{\,}_{Mi}}{H}_*(N,G)$ is the Milnor homology \cite{10}.\\

{\bf Corollary 6.}
	{\rm a)} If $X$ is a locally compact {\color{black}Hausdorff} space with second countable axiom, then for each countable system $\{N_i\}$ of closed neighborhoods of a closed subspace  $A$ of  $X$ there is a short exact sequence 
	$$
	0 \lra \llm^{(1)} \os{{\,}_M}{H}_{n+1} (N_i) \lra \os{{\,}_M}{H}_n(A,G) \lra \llm \os{{\,}_M}{H}_n(N_i) \lra 0,
	$$
	where $\os{{\,}_M}{H}_*$ is the Massey homology \cite{9}.
	
	{\rm b)} If $X$ is a locally compact, paracompact Hausdorff space with second countable axiom, then for each countable system $\{N_i\}$ of closed neighborhoods of a closed subspace $A$ of $X$ there is a short exact sequence 
	$$
	0 \lra \llm^{(1)} \os{{\,}_K}{H}_{n+1} (N_i) \lra \os{{\,}_K}{H}_n(A,G) \lra \llm \os{{\,}_K}{H}_n(N_i) \lra 0,
	$$
	where $\os{{\,}_K}{H}_*$ is the Kolmogoroff homology \cite{7}.
	
	{\rm c)} If $X$ is a compact Hausdorff space with second countable axiom, then for each countable system $\{N_i\}$ of closed neighborhoods of a closed subspace $A$ of $X$ there is a short exact sequence
	$$
	0 \lra \llm^{(1)} \os{{\,}_{Mi}}{H}_{n+1} (N_i) \lra \os{{\,}_{Mi}}{H}_n(A,G) \lra \llm \os{{\,}_{Mi}}{H}_n(N_i) \lra 0,
	$$
	where $\os{{\,}_{Mi}}{H}_*$ is the Milnor homology \cite{10}.
	
	{\rm d)} If $X$ is a compact metric space, then for each countable system $\{N_i\}$ of a closed neighborhoods  of closed subspace $A$ of $X$ there is a short exact sequence
	$$
	0 \lra \llm^{(1)} \os{st}{H}_{n+1} (N_i) \lra \os{st}{H}_n(A,G) \lra \llm \os{st}{H}_n(N_i) \lra 0,
	$$
	where $\os{st}{H}_*$ is the  Steenrod homology \cite{12}.
}

\paragraph*{\bf 1 Tautness} In the paper \cite[\S1.1]{9}  W. Massey defined {\color{black}the} cochain complex $C_c^*(X,G)$ for any locally compact Hausdorff spaces $X$ and any abelian group  $G$. By {\color{black}Theorem 4.1 \citep[\S 4.4]{9},} for each locally compact Hausdorff space $X$ and each integer $n$ the cochain group $C_c^n(X,\bZ)$ with integer coefficient is a free abelian group. The chain complex $C_*(X,G)=\Hom (C_c^*(X),G)$ is completely defined by {\color{black}the cochain complex} $C_c^*(X)$ with coefficient group $\mathbb{Z}$ of integers and therefore, by {\color{black}Theorem 4.1} \citep[\S 4.4]{9} and {\color{black}Theorem 4.1 (Universal Coefficients) \citep[\S III.4]{8}, there is {\color{black}an} exact sequence \citep[Corollary 4.18, \S 4.8]{9}}

\begin{equation}\label{eq1}
0 \lra \Ext(H_c^{n+1}(X),G) \lra \os{{\,}_M}{H}_n(X,G) \lra \Hom(H_c^n(X),G) \lra 0,
\end{equation}
where $\os{{\,}_M}{H}_n(X,G)$ is {\color{black}the Massey homology group and ${H}^{n+1}_c(X,G)$ is the Massey cohomology group, respectively \cite[\S 4.6]{9}, i.e $\os{{\,}_M}{H}_n(X,G)=H_n(\Hom(C_c^*(X),G))$ and ${H}^{n+1}_c(X)=H_n(C_c^*(X,\mathbb{Z})$. Moreover, this sequence is split. However, the splitting is only natural with respect to coefficient homomorphisms.}

Let $X$ be a locally compact space and $A$ be a closed subspace of $X$. In this case, for each closed neighborhood $N$ of $A$  there is a homomorphism $i_N:h^n(N)\to h^n(A)$. If $N_1\subset N_2$, then there is a homomorphism $i_{N_1,N_2}:h^n(N_2)\to h^n(N_1)$. Therefore, {\color{black} there is the direct system $\{h^n(N)\}$ of abelian groups and homomorphisms $\{i_{N_1,N_2}\}$.} Consequently, there exists a natural homomorphism
$$
i^n:\rlm h^n(N)\lra h^n(A).
$$
If $h^*=H_c^*$ is the Massey cohomology {\color{black} \cite[\S 4.6]{9}, then (see Theorem 6.4 \cite[\S 6.4]{9})} there is an isomorphism
\begin{equation}\label{eq12}
i^n:\rlm H_c^n(N,G) \os{\sim}{\lra} H_c^n(A,G).
\end{equation}
In this case, a subspace $A$ is said to be 	taut with respect to cohomology theory $H^*_c(-,G)$.

Let $h_*$ be a homology theory on the category of some topological spaces. Let $A$ be a closed subspace of $X$. In this case, for a neighborhood $N$ of $A$  there is a homomorphism $i_N:h_n(A)\to h_n(N)$. If $N_1\subset N_2$, then there is a homomorphism $i_{N_1,N_2}:h_n(N_1)\to h_n(N_2)$. Therefore, {\color{black} there is the inverse system $\{h_n(N)\}$ of abelian groups and homomorphisms $\{i_{N_1,N_2}\}$.} Consequently, there exists a natural homomorphism
$$
i_n: h_n(A) \os{}{\lra} \llm h_n(N) .
$$
{\color{black}
\begin{definition}
	A closed subspace $A$ of a space $X$ is said to be tautly embedded in $X$, if for some set $N$ of neighborhoods there exists a long exact sequence 
	\begin{gather*}
    \cdots \lra \llm^{(2k+1)} h_{n+k+1}(N) \lra \cdots \lra \llm^{(3)} h_{n+2}(N) \lra \llm^{(1)} h_{n+1}(N) \lra \\
     \lra h_{n}(A,G) \os{i_n}{\lra} \llm h_{n}(N) \lra  \llm^{(2)} h_{n+1}(N) \lra \cdots \lra \llm^{(2k)} h_{n+k}(N) \lra \cdots\ ,
     \end{gather*}
	which contains the homomorphism  $h_n(A) \os{i_n}{\lra}\llm h_n(N)$. 
\end{definition}
}
Let $\os{{\,}_M}{H}_n(X,G) =H_n(\Hom(C_c^*(X),G))$ be the Massey homology group of locally compact Hausdorff spaces. Let $A$ be a closed subspace of $X$ and $N$ be the set of all closed neighborhoods of $A$. Then each homomorphism $i_{N_1,N_2}:N_1\to N_2$ is a proper map (a map is proper if it is continuous and if inverse image of any compact subspace is compact) and induces a homomorphism $i_{N_1,N_2}:\os{{\,}_M}{H}_n(N_1)\to \os{{\,}_M}{H}_n(N_2),$
which defines the homomorphism 
$$
i_*:\os{{\,}_M}{H}_n(A,G) \lra \llm \os{{\,}_M}{H}_n(N,G).
$$
Since the short exact sequence \eqref{eq1} is natural, there is a commutative diagram 
\begin{gather}
\xymatrix{
	0 \ar[r] & \Ext(H_c^{n+1}(A),G) \ar[r] \ar[d]^-{\rho'} & \os{{\,}_M}{H}_n(A,G) \ar[r] \ar[d]^-{i_n} & \Hom(H_c^n(A),G) \ar[r] \ar[d]^-{\rho''} & 0 \\
	0 \ar[r] & \llm \Ext(H_c^{n+1}(N),G) \ar[r] & \llm \os{{\,}_M}{H}_n(N,G) \ar[r] & \llm \Hom(H_c^n(N),G) \ar[r]  & ~~ & \
} \nonumber \\
 \lra \llm^{(1)} \Ext(H_c^{n+1}(N),G) \lra \llm^{(1)} \os{{\,}_M}{H}_n(N,G) \lra \llm^{(1)} \Hom(H_c^n(N),G) \lra \cdots  \label{eq13}
\end{gather}
with exact arrows.

Using the isomorphism \eqref{eq12} and properties of functors $\Hom(-,G)$ and  $\rlm$, there is an isomorphism 
\begin{equation}\label{eq14}
\Hom(H_c^n(A),G)\approx \Hom (\rlm H_c^n(N),G) \approx \llm \Hom(H_c^n(N),G).
\end{equation}
Therefore, a homomorphism $\rho''$ is an isomorphism.

Using the isomorphism \eqref{eq14} and the commutative diagram \eqref{eq13}, we obtain the following commutative diagram
\begin{equation}\label{eq15}
\xymatrix{
	& 0 \ar[d] & 0 \ar[d] \\
	& \Ker \rho' \ar[r]^-{\sim} \ar[d] & \Ker i_n \ar[d] \\
	0 \ar[r] & \Ext(H_c^{n+1}(A),G) \ar[r] \ar[d]^-{\rho '} & \os{{\,}_M}{H}_n(A,G) \ar[r] \ar[d]^-{i_n}
	& \Hom(H_c^n(A),G) \ar[r] \ar[d]^-{\overset{~~}{\simeq}}
	& 0 \\
	0 \ar[r] & \llm \Ext(H_c^{n+1}(N),G) \ar[r] \ar[d] & \llm \os{{\,}_M}{H}_n(N,G) \ar[r] \ar[d]
	& \llm \Hom(H_c^n(N),G) \ar[r]  & 0 \\
	& \Coker \rho' \ar[r]^-{\sim} \ar[d] & \Coker i_n \ar[d] \\
	& 0 & 0
}
\end{equation}

By {\color{black}Lemma 1 \cite{11n},} if a complex $C_*$ is free, then there is an exact sequence 
$$
0 \lra \Hom(B_{n-1},G) \lra Z^n \lra \Hom(H_n,G) \lra 0,
$$
where {\color{black}$B_{n-1}=\text{Im}~ \partial_{n},~~ \partial:C_{n} \to C_{n-1}$ and  $Z^n=\Ker \dl^{n+1}$, $\dl^{n+1}:C^n \to C^{n+1}$, where $C^*=Hom(C_*,G)$. In our case, we have a dual version. In particular,} the cochain complex $C^*_c(-)$ is free and hence, there is an exact sequence
$$
0 \lra \Hom(B_c^{n+1}(-),G) \lra Z_n \lra \Hom(H_c^n(-),G) \lra 0,
$$
where {\color{black}$Z_n=\Ker \partial_{n}$,  $\partial_{n}:C_n \to C_{n-1}$ and $B_c^{n+1} =\IIm \dl^n$, $\dl^n:C_c^{n} \to C_c^{n+1}$, where $C_*=Hom(C^*_c,G)$}. Consequently, using {\color{black}Lemma 2 \citep[]{11n},} for each $A \subset N $ there is a commutative diagram with exact arrows
\begin{equation}\label{eq16}
\xymatrix{
	0 \ar[r] & \Hom(B_c^{n+1}(N),G) \ar[r]\ar[d] & Z_n \ar[r]\ar[d] & \Hom (H_c^n(N),G) \ar[r]\ar@{=}[d] & 0 \\
	0 \ar[r] & \Ext(H_c^{n+1}(N),G) \ar[r] & H_n(N,G) \ar[r] & \Hom(H_c^n(N),G) \ar[r] & 0\,.
}
\end{equation}

\begin{theorem}\label{thm4}
	Let $\{C_c^*(N)\}$ be a direct system of free chain complexes $C_c^*(N)$ of closed neighborhoods $N$ of a closed subspace $A$ of locally compact Hausdorff spaces $X$ and $G$ be an abelian group. In this case, for each $n\in \bZ$, and $i\geq 1$ there is a short exact sequence 
	\begin{equation}\label{eq17}
	0 \lra \llm^{(i)} \Ext(H_c^{n+1}(N),G) \lra \llm^{(i)} \os{{\,}_M}{H}_n(N,G) \lra \llm^{(i)} \Hom(H_c^*(N),G) \lra 0,
	\end{equation}
	which splits for $i\geq 2$.
\end{theorem}

\begin{proof} Using the split sequence \eqref{eq1} and commutative diagram \eqref{eq16}, we obtain the following commutative diagram with the exact arrows 
	\begin{equation}\label{eq18} \fontsize{10}{12pt}\selectfont 
	\xymatrix{
		\cdots \ar[r] & \llm^{(i)}\Hom(B_c^{n+1}(N),G) \ar[r]\ar[d] & \llm^{(i)} Z_n \ar[r]\ar[d] & \llm^{(i)} \Hom (H_c^n(N),G) \ar[r]\ar@{=}[d] & \cdots \\
		\cdots \ar[r] & \llm^{(i)}\Ext(H_c^{n+1}(N),G) \ar[r] & \llm^{(i)} \os{{\,}_M}{H}_n(N,G) \ar[r] & \llm^{(i)}\Hom(H_c^n(N),G) \ar[r] & \cdots\,.
	}
	\end{equation}
	In the paper \cite{6n} it is shown that for each direct system  $\{A_\al\}$ of abelian groups $A_\al$ there exists {\color{black}a} short exact sequence
	\begin{gather}
	0 \lra \llm^{(1)} \Hom(A_\al,G) \lra \Ext(\rlm A_\al,G) \lra \llm \Ext(A_\al,G) \nonumber \lra \\
	\lra \llm^{(2)} \Hom(A_\al,G) \lra 0,
	\label{eq19}
	\end{gather}
	and for each $i\geq 1$ there is an isomorphism 
	\begin{equation}\label{eq20}
	\llm^{(i)} \Ext (A_\al,G) \approx \llm^{(i+2)} \Hom(A_\al,G).
	\end{equation}
	
	Consider a direct system $\{B_c^{n+1}(N)\}$ of free groups $B_c^{n+1}(N)$. In this case, by the exact sequence  \eqref{eq19} and the isomorphism  \eqref{eq20} we have	
	\begin{equation}\label{eq21}
	\llm^{(i)} \Hom(B_c^{n+1}(N),G)=0 \quad \text{for} \;\; i\geq 2.
	\end{equation}
	By the diagram \eqref{eq18} and the equality \eqref{eq21} we have 
	\begin{itemize}
		\item[a)] an isomorphism $\llm^{(i)} Z_n \approx \llm^{(i)} \Hom(H_c^n(N),G)$ for each  $i\geq 2$;
		
		\item[b)] an epimorphism  $\llm^{(i)} \os{{\,}_M}{H}_n(N,G) \lra \llm^{(i)} \Hom(H_c^n(N),G)$ for each $i\geq 1$;
		
		\item[c)] a monomorphism $\llm^{(i)} \Ext(H_c^{n+1}(N),G) \lra \llm^{(i)} \os{{\,}_M}{H}_n(N,G)$ for each $i\geq 2$;
		
		\item[d)] the trivial homomorphism $\llm^{(i)} \Hom(H_c^n(N),G) \lra \llm^{(i+1)} \Ext(H_c^{n+1}(N),G)$ for each \mbox{$i\geq 1$.}
	\end{itemize}
	
	By  b) and  c) for each $i\geq 2$ we have a short exact sequence
	\begin{equation}\label{eq22}
	0 \lra \llm^{(i)} \Ext(H_c^{n+1}(N),G) \lra \llm^{(i)} \os{{\,}_M}{H}_n(N,G) \lra \llm^{(i)} \Hom(H_c^n(N),G) \lra 0.
	\end{equation}
	{\color{black}On the other hand, by a)} for each $i\geq 2$ we can define a homomorphism 
	$$
	\llm^{(i)} \Hom(H_c^n(N),G) \os{\sim}{\lra} \llm^{(i)} Z_n \lra \llm^{(i)} \Hom(H_c^n(N),G).
	$$
	It is clear that the composition 
	$$
	\llm^{(i)} \Hom(H_c^n(N),G) \os{\sim}{\lra} \llm^{(i)} Z_n \lra \llm^{(i)} \os{{\,}_M}{H}_n(N,G) \lra \llm^{(i)}  \Hom(H_c^n(N),G)
	$$
	is the identity map. Therefore, for each $i\geq 2$ the sequence \eqref{eq22} splits. 
\end{proof}

\begin{theorem}\label{thm5}
	The system $\{N\}$ of closed neighborhoods $N$ of closed subspace $A$ of a locally compact Hausdorff space $X$, directed by inclusion, induces the following exact sequence
	\begin{gather*}
	\cdots \lra \llm^{(2k+1)} \os{{\,}_M}{H}_{n+k+1}(N) \lra \cdots \lra \llm^{(3)} \os{{\,}_M}{H}_{n+2}(N) \lra \llm^{(1)} \os{{\,}_M}{H}_{n+1}(N) \lra\\
	\lra \os{{\,}_M}{H}_{n}(A) \os{i_n}{\lra} \llm \os{{\,}_M}{H}_{n}(N) \lra  \llm^{(2)} \os{{\,}_M}{H}_{n+1}(N) \lra \cdots \lra \llm^{(2k)} \os{{\,}_M}{H}_{n+k}(N) \lra \cdots\,,
	\end{gather*}
	where $\os{{\,}_M}{H}_*(-) =\os{{\,}_M}{H}_*(-,G)$ is {\color{black}the Massey homology} with a coefficient abelian group $G$.
\end{theorem}

\begin{proof}
	By the diagram \eqref{eq15} and the property d) from  Theorem \ref{thm4} we have the following exact sequence
	\begin{gather}
	0 \lra \llm \Ext(H_c^{n+1}(N),G) \lra \llm \os{{\,}_M}{H}_n(N,G) \lra \llm \Hom(H_c^n(N),G) \nonumber \lra \\
	\lra \llm^{(1)} \Ext(H_c^{n+1}(N),G) \lra \llm^{(1)} \os{{\,}_M}{H}_n(N,G) \lra \llm^{(1)} \Hom(H_c^n(N),G) \lra 0. \label{eq21-1}
	\end{gather}
	By the isomorphism \eqref{eq14} and exact commutative diagram \eqref{eq13}, we will obtain the following exact sequence
	\begin{equation}\label{eq22-1}
	0 \lra \llm^{(1)} \Ext(H_c^{n+1}(N),G) \lra \llm^{(1)} \os{{\,}_M}{H}_n(N) \lra \llm^{(1)} \Hom(\os{{\,}_M}{H}{}_c^n(N),G) \lra 0,
	\end{equation}
	where by Theorem \ref{thm4} the sequence \eqref{eq22-1} splits for each $i\ge2$.
	
	Note that, by (5) and (9), there is an exact sequence	\begin{gather}
	0 \lra \llm^{(1)} \Hom(H_c^{n+1}(N),G) \lra \os{{\,}_M}{H}_n(A,G) \lra \llm \os{{\,}_M}{H}_n(N,G) \nonumber \lra \\
	\lra \llm^{(2)} \Hom(H_c^{n+1}(N),G) \lra 0.
	\label{eq23}
	\end{gather}
	By the exact sequences  \eqref{eq17}, \eqref{eq22-1} and \eqref{eq23} for $i\geq 1$, an isomorphism \eqref{eq20} and Theorem \ref{thm4} we have the exact sequence
	
	\[ 
	\xymatrix{
		&& \vdots \ar@{-->}[d] \\
		0 \ar[r] &\llm^{(3)} \Ext (H_c^{n+3}(N),G) \ar[r] & \llm^{(3)} \os{{\,}_M}{H}_{n+2}(N)  \ar[r] \ar@{-->}[d]
		& \llm^{(3)} \Hom (H_c^{n+2}(N),G) \ar[r] \ar[d]^-{\overset{~~}{\simeq}}  & 0 \\
		0 &\llm^{(1)} \Hom (H_c^{n+1}(N),G) \ar[l] \ar@{=}[d] & \llm^{(1)} \os{{\,}_M}{H}_{n+1}(N)  \ar[l] \ar@{-->}[d]
		& \llm^{(1)} \Ext (H_c^{n+2}(N),G) \ar[l] & 0 \ar[l] \\
		0 \ar[r] & \llm^{(1)} \Hom (H_c^{n+1}(N),G) \ar[r] & \os{{\,}_M}{H}_{n}(A,G)  \ar[d] \\
		&& \llm \os{{\,}_M}{H}_{n}(N,G)  \ar[r] \ar@{-->}[d] & \llm^{(2)} \Hom (H_c^{n+1}(N),G) \ar[r] \ar@{=}[d] & 0 \\
		0 &\llm^{(2)} \Ext (H_c^{n+2}(N),G) \ar[l] \ar[d]^-{\overset{~~}{\simeq}}
		& \llm^{(2)} \os{{\,}_M}{H}_{n+1}(N)  \ar[l] \ar@{-->}[d]
		& \llm^{(2)} \Hom (H_c^{n+1}(N),G) \ar[l] & 0 \ar[l] \\
		0 \ar[r] &\llm^{(4)} \Hom (H_c^{n+2}(N),G) \ar[r] & \llm^{(4)} \os{{\,}_M}{H}_{n+2}(N)  \ar[r] \ar@{-->}[d]
		& \llm^{(4)} \Ext (H_c^{n+3}(N),G) \ar[r] & 0  \\
		&& \vdots
	}
	\]
\end{proof}

\paragraph*{ \bf 2 The Kolmogoroff homology}
Our aim is to study the tautness property for other exact homology theories \cite{7}, \cite{12}, \cite{3}, \cite{10}. Among them, one of the main places is taken by {\color{black}the} Kolmogoroff homology, which was defined as earlyer as in 1936  \cite{7}, {\color{black}\cite{15}.} A. N. Kolmogoroff defined homology on the category of locally compact Hausdorff spaces and proper maps with a compact coefficient group \cite{7}, {\color{black}\cite{15}}. Using the homology defined by all finite partitions, in the paper \citep{3} G. S. Chogoshvili proved that Kolmogoroff homology and {\color{black}the} Alexandroff-\v{C}ech homology groups are isomorphic {\color{black}on the category of compact Hausdorff spaces} for a compact coefficient group {\color{black}\cite{15}}. Since the Steenrod and {\color{black}the} Alexandroff-\v{C}ech homologies are isomorphic on the category of compact metric spaces {\color{black}for a compact coefficient group }\cite{12}, we have the isomorphisms 

\begin{equation}\label{eq61}
\os{{\,}_K}{H}_*(X,G) \os{1}{\approx} \os{ch}{H}_*(X,p,G) \os{2}{\approx} \os{ch}{H}_*(X,sp,G) \os{3}{\approx} \check{H}_*(X,G) \os{4}{\approx} \os{st}{H}_*(X,G),
\end{equation}
where $\os{K}{H}_*(-,G)$ is the Kolmogoroff \cite{7}, \cite{15}, $\os{ch}{H}_*(-,p,G)$ is the Chogoshvili projective \citep{3}, {\color{black}\cite{15}}, $\os{ch}{H}_*(-,sp,G)$ is the Chogoshvili spectral  \citep{3}, {\color{black}\cite{15}}, $\check{H}_*(-,G)$ is the Alexandroff-\v{C}ech  \cite{4} and $\os{st}{H}_*(-,G)$ is the Steenrod  \cite{12} homology theories.
Later the Kolmogoroff and the Chogoshvili homology theories were generalized and defined even for a discrete coefficient groups \cite{15}. However, there are no isomorphisms 2 and 4 as in \eqref{eq61} {\color{black}\cite{15}}.
Consequently, there was a natural interest to find the connection between the Kolmogoroff  and Steenrod homology groups for any discrete groups. Using the {\color{black}Uniqueness Theorem} given by {\color{black}Milnor \citep{10},} in the paper \cite{2} it is proved that on the category of compact metric spaces {\color{black}the Kolmogoroff and the Steenrod homologies} are isomorphic {\color{black}even} for {\color{black}any} discrete coefficient groups {\color{black}\cite{15}. Therefore, to study tautness properties for an exact homology theory, it is crucial to find a connection between the Kolmogoroff and the Massey homology theories.}

By {\color{black}Theorem 2.8} \citep[\S 2.2]{9}, if $X$ be a locally compact Hausdorff {\color{black}noncompact} space and $\dot{X}$ its one point Alexandroff compactification, then the inclusion $\mu:X\to \dot{X}$ induces an isomorphism 
\begin{equation}\label{eq2}
\mu^*: H_c^q(X,G) \os{\sim}{\lra} {H}_c^q(\dot{X},*,G).
\end{equation}

\begin{corollary}\label{c2}
The inclusion $\rho:X\to (\dot{X},*)$, where $\dot{X}$ is the one point Alexandroff compactification  of locally compact Hausdorff space $X$, indices an isomorphism 
\begin{equation}\label{eq5}
  \rho_*: \os{{\,}_M}{H}_*(X,G) \os{\sim}{\lra} \os{{\,}_M}{H}_*(\dot{X},*,G).
\end{equation}
\end{corollary}

 \begin{proof}
The inclusion $\rho:X\to (\dot{X},*)$ induces a commutative diagram with the exact sequences  
$$
\xymatrix{
0 \ar[r] & \Ext(H_c^{n+1}(X),G) \ar[r] \ar[d]^-{\rho'} & \os{{\,}_M}{H}_n(X,G) \ar[r] \ar[d]^-{\rho_n} & \Hom(H_c^n(X),G) \ar[r] \ar[d]^-{\rho''} & 0 \\
0 \ar[r] & \Ext({H}_c^{n+1}(\dot{X},*),G) \ar[r]  & \os{{\,}_M}{H}_n(\dot{X},*,G) \ar[r]  & \Hom({H}_c^n(\dot{X},*),G) \ar[r] & 0\,.
}
$$
By the isomorphism \eqref{eq2}, the homomorphisms $\rho'$ and $\rho''$ are isomorphisms as well. Therefore, by the Lemma of Five Homomorphisms we obtain the required statement.
\end{proof}

{\color{black}Now we will define the Kolmogoroff homology theory and using the isomorphism \eqref{eq5}, we will find the connection of it with the Massey homology theory.

Let $X$ be a locally compact Hausdorff space. A subset $A$ of space $X$ is called bounded if $\bar{A}$ is compact {\cite[Definition 6.1, \S X.6]{4}}.}

\begin{definition} Let $X$ be a {\color{black}locally} compact space, $E_X$ be {\color{black}the} set of all bounded subsets $E_i$ of $X$, $G$ be an abelian group. Denote by $E_X^{n+1}=E_X \times E_X \times \dots  \times E_X$ - a direct product of $E_X$. An $n$-dimension Kolmogoroff chain of the space $X$ is called a function $f_n:E_X^{n+1} \to G$ satisfying the following conditions: 
\begin{itemize}
\item[$K1)$] If $E_i=E_i'\cup E_i''$ and $E_i'\cap E_i''=\vnth$, then
$$
f_n(E_0,\dots,E_i,\dots,E_n) =f_n(E_0,\dots,E_i',\dots,E_n) +f_n(E_0,\dots,E_i'',\dots,E_n);
$$

\item[$K2)$] $f_n$ will not change by even permutation  and will changes just the sign by odd permutation of argument; $f_n=0$, if two arguments are the same;

\item[$K3)$] If $\ol{E}_0 \cap \cdots \cap \ol{E}_n=\vnth$, then $f_n(E_0,\dots,E_n)=0$.
\end{itemize}
\end{definition}

The sum $f_n'+f_n''$ of two $f_n'$, $f_n''$ function{\color{black}s} is defined  by the following equation
$$
(f_n'+f_n'')(E_0,\dots,E_n)=f_n'(E_0,\dots,E_n)+f_n''(E_0,\dots,E_n).
$$
It is clear that {\color{black} the set of} all $n$-dimensional functions
$f_n$ is an abelian group, which is {\color{black}denoted} by $K_n(X,G)$.
The {\color{black}boundary} operator $\Dl:K_n(X,G) \to K_{n-1}(X,G)$ is defined by the equation 
$$
\Dl f_n(E_0,\dots,E_{n-1}) =f_n(U,E_0,\dots,E_{n-1}),
$$
where $U$ is an open bounded subset which includes  $\bigcup\limits_{i=1}^{n-1} \ol{E}_i$. Since the space $X$ is locally compact, such $U$ exists and the {\color{black}boundary} operator $\Dl$ does not {\color{black}depend} on the choice of $U$.

The homology of the chain complex $K_*(X,G) =\{K_n(X,G),\Dl\}$ is called {\color{black}the Kolmogoroff } homology of a locally compact space $X$ and it is denoted by $\os{{\,}_K}{H}_*(X,G)$.

\begin{definition}
A locally finite system of {\color{black}bounded} subspaces $e_i$ of space $X$, which are pairwise {\color{black}non-intersecting} and the sum of them that gives the {\color{black}whole} space  $X=\cup e_i$ is called a regular partition.
\end{definition}

\begin{lemma}\label{l1}
For each locally compact, paracompact space $X$ there exits a regular partition.
\end{lemma}

\begin{proof}
Since $X$ is a locally compact space, for each point  $x\in X$  there exists a bounded neighborhood $U_x$.
Since space $X$ is paracompact as well, an open covering $\{U_x\}_{x\in X}$ has a locally finite  refinement $\{O_\lb\}$, which is contained with bounded subspaces $O_\lb$. If we write the elements of the covering $\{O_\lb\}$ as a transfinite sequence $O_1,O_2,\dots,O_\lb,\dots$, then we construct a regular covering in the following way: $O_1,O_2\setminus O_1,\dots,O_\lb\setminus \bigcup\limits_{i<\lb} O_i$, where  $O_i$  runs through all  the ordinal numbers preceding $\lb$.
\end{proof}

Denote by $S=\{S_\al\}$ {\color{black}the} system of all regular partitions $S_\al$ of a space $X$.

\begin{lemma}\label{l2}
Each compact subspace $F$ of a locally compact space has a nonempty intersection only with finite number of closures  $e_i^\al\in S_\al$.
\end{lemma}

\begin{proof}
Since $S_\al$ is {\color{black}a} locally finite system, for each point $x\in F$ there exists neighborhood $U_x$, which has a nonempty intersection only with finite many elements $e_i^\al\in S_\al$. From the collection $\{U_x\}_{x\in F}$ of the neighborhoods  a finite subsystem can be chosen, union of which covers the space $F$. Since for each open subspace $U$ and subspace $B$ there is an equivalency  $U\cap B \neq \vnth \Leftrightarrow U\cap \ol{B}\neq \vnth$, we obtain truthfulness of the lemma.
\end{proof}
Denote by $N_\al$ the nerve of a regular partition $S_\al\in S$, which consists of simplexes $\sg^n=(e_0^\al,\dots,e_n^\al)$, for which $\cap \ol{e}{}_i^\al\neq \vnth$. By the lemma \ref{l2} the nerve $N_\al$ is locally finite \cite{3}, \cite{15}.

If $S_\al<S_\bt$, i.e. $S_\bt$ is refinement of $S_\al$ and if for each vertex  $e_j^\bt\in N_\bt$ we take the uniquely defined vertex $e_i^\al\in N_\al$, which contains $e_j^\bt$, then we obtain a simplicial map $\pi_{\bt\al}:N_\bt\to N_\al$. By Lemma \ref{l2}, the map $\pi_{\bt\al}$ will be locally finite \cite{4}, i.e. inverse image of each simplex contains only finite many numbers of simplexes.

If we take for each $S_\al\in S$ the group of the infinite chains $C_n^{inf}(N_\al,G)$ of nerve $N_\al$ and homomorphisms $\pi_{\bt\al}^*:C_n^{inf}(N_\bt,G) \to C_n^{inf}(N_\al,G)$, induced by simplicial maps $\pi_{\bt\al}$, then we obtain {\color{black}an} inverse system $\{C_n^{inf}(N_\al,G),\pi_{\bt\al}^*\}$, the inverse limit group of which is denoted by 
$$
C_n^{inf}(X,G) =\llm \{C_n^{inf} (N_\al,G),\pi_{\bt\al}^*\}.
$$

The {\color{black}boundary operator} $\pa :C_n^{inf}(X,G) \to C_{n-1}^{inf} (X,G)$ is defined by {\color{black}the boundary operators} $\pa_\al:C_n^{inf} (N_\al,G) \to C_{n-1}^{inf}(N_\al,G)$, which commute with homomorphisms ~$\pi_{\bt\al}^*$. The homology group of the obtained complex $C_*^{inf}(X,G)$ is called the Chogoshvili {\color{black}projection} homology group and s denoted by $\os{ch}{H}_*(X,p,G)$.

\begin{definition} Let $A=\{A_i\}$ and $B=\{B_j\}$ be finite systems of sets such that $B=\{B_j\}$ consists of pairwise non-intersecting sets. We will say that a system $B$ is a mosaic of the system $A$, if for each $B_j\in B$ there exists $A_i\in A$ such that $B_j\sbs A_i$ and $A_i=\bigcup\limits_j B_{i_j}$, where $B_{i_j}\in B$.

\end{definition}

\begin{lemma}\label{l3}
 For each finite system $A=\{A_i\}$, $i=0,\dots,n$ of sets $A_i$ there exists a mosaic.
\end{lemma}

\begin{proof}
The system consisting of subspaces $\bigcap\limits_{i=0}^n A_i$, $\bigcap\limits_{t=1}^n A_t \setminus \bigcup\limits_{i_*\neq i_t} A_{i_*}$, where $i_1,\dots,i_p$ -- $p$ are different indexes from the system  $i=0,\dots,n$, $1\leq p\leq n$ and  $i_*$ obtains all value in the same system, except $i_1,\dots,i_p$ which is a mosaic.
\end{proof}

\begin{lemma}\label{l4}
If $\wt{f}_n$ is a function on the directed system  $e_0,\dots,e_n$ mutually non-intersecting {\color{black}bounded} subspaces $e_i$ of locally compact space $X$, which satisfies the conditions $K1)$--$K3)$, then it can be extended to the function $f_n\in K_n(X,G)$.
\end{lemma}

\begin{proof}
By the lemma \ref{l3}, for each directed system $E_0,\dots,E_n$ of {\color{black}bounded} subspaces $E_i$ there exists a mosaic $\{e_{i_j}\}$ such that $E_i=\cup e_{i_j}$.  Therefore, the function $f_n(E_0,\dots,E_n)=\sum\limits_{0_j,\dots,n_j} \wt{f}_n(e_{0_j},\dots,e_{n_j})$ does not depend on the choice of mosaic. Indeed, let $\{e_{i_j}'\}$ be another  mosaic of the system  
$E_0,\dots,E_n$ and $f_n'(E_0,\dots,E_n)=\sum\limits_{0_j,\dots,n_j}\wt{f}_n(e_{0_j}',\dots,e_{n_j'})$. It is clear that the intersection $\{e_{i_j}\}\wedge \{e_{i_j}'\}=\{e_{i_j}''\}$ is a mosaic not only for $\{E_i\}$, but for each  mosaic. Therefore, we have 
\begin{align*}
f_n(E_0,\dots,E_n)& =\sum \wt{f}_n(e_{0_j},\dots,e_{n_j})=\sum\wt{f}_n(e_{0_j}'',\dots,e_{n_j}'')\\
& =\sum \wt{f}_n(e_{0_j}',\dots,e_{n_j}')=f_n'(e_0,\dots,e_n).
\end{align*}
Thus, the defined function $f_n$, satisfies the conditions  $K1)$--$K3)$ and so $f_n\in K_n(X,G)$.
\end{proof}

\begin{theorem}\label{t1}
Let $X$ be a locally compact, paracompact Hausdorff space. Then the Kolmogoroff homology $\os{{\,}_K}{H}_*(X,G)$ is isomorphic to the Chogoshvili {\color{black}projection} homology  $\os{ch}H_*(X,p,G)$.
\end{theorem}

\begin{proof}
We will prove much stronger statement. In particular, there is an isomorphism of chain complexes $K_*(X,G)$ and $C_*^{inf}(X,G)$,
\begin{equation}\label{eq7}
  K_*(X,G) \approx C_*^{inf}(X,G).
\end{equation}

For each $S_\al\in S$ define a homomorphism  $\xi_\al:K_*(X,G) \to C_*^{inf}(N_\al,G)$ by the formula   $\xi_\al f_n(e_0^\al,\dots,e_n^\al)=f_n(e_0^\al,\dots,e_n^\al)$, where $f_n\in K_n(X,G)$, $(e_0^\al,\dots,e_n^\al)\in N_\al$. Therefore
$$
\xi_\al\Dl f_n(e_0^\al,\dots,e_{n-1}^\al)=\Dl f_n(e_0^\al,\dots,e_{n-1}^\al)=f_n(U,e_0^\al,\dots,e_{n-1}^\al).
$$
By Lemma \ref{l2} and the property of uniqueness of a function $f_n$ (property $K1)$), we have
$$
f_n(U,e_0^\al,\dots,e_{n-1}^\al)=f_n(\cup e_{i_t}^\al,e_0^\al,\dots,e_{n-1}^\al)=
    \sum_{i_t} f_n(e_{i_t}^\al,e_0^\al,\dots,e_{n-1}^\al).
$$
Therefore,
$$
\xi_\al \Dl f_n(e_0^\al,\dots,e_{n-1}^\al) =\sum_{i_t} f_n(e_{i_t}^\al,e_0^\al,\dots,e_{n-1}^\al)=
    \pa_\al f_n(e_0^\al,\dots,e_{n-1}^\al),
$$
i.e. $\xi_\al \Dl=\pa_\al \xi_\al$.

A homomorphism $\xi_\al$ induces an isomorphism 
$$
\xi:K_*(X,G) \lra C_*^{inf} (X,G).
$$
Let $c_n=\{c_{\al,n}\}\in C_n^{inf}(X,G)$ and $E_0,\dots,E_n$ be a system of mutually non-intersecting bounded subspaces. If we add to this system the subspace $X\setminus \bigcup\limits_i E_i$, then we obtain a finite partition $D$ of space $X$. Let $S_\al\in S$, then  $D\wedge S_\al=S_{\al'}\in S$ and for each $E_i=\cup e_{i_j}^{\al'}$, where $e_{i_j}^{\al'}\in S_{\al'}$.

Let $\wt{f}_n$ be a function on the system $E_0,\dots,E_n$ which is defined by 
$$
\wt{f}_n(E_0,\dots,E_n)=\sum\limits_{0_i,\dots,n_j} c_{\al',n}(e_{0_j}^{\al'},\dots,e_{n_j}^{\al'}),
$$
where $0_j,\dots,n_j$ obtain all values, where $(e_{0_j}^{\al'},\dots,e_{n_j}^{\al'})$ denotes a simplex in $N_{\al'}$. It is easy to show that such defined function $\wt{f}_n$ does not depend on the choice of $S_\al$ and it satisfies the properties $K1)$--$K3)$. By Lemma \ref{l4} a function $\wt{f}_n$ can be extended to a function  $f_n\in K_n(X,G)$. If we define a homomorphism 
$$
\eta: C_*^{inf}(X,G) \lra K_*(X,G)
$$
by $\eta(c_n)=f_n$, then it will be inverse of the homomorphism $\xi$.
\end{proof}

\begin{theorem}\label{t2}
Let $\{G_\al,p_{\bt\al}\}_{\al\in \Lb}$ be a direct system of {\color{black}free abelian} groups $G_\al$, which satisfies the following conditions:
\begin{itemize}
\item[1)] For each group $G_\al$ there exists a base $B=\{g_1^\al,g_2^\al,\dots,g_\tau^\al, \dots\}$;

\item[2)] For each pair $\al<\bt$, $\al,\bt\in \Lb$, a set of indexes $\{1,2,\dots,\tau(\bt),\dots\}$ of elements of base  $B_\bt$ can be decomposed with non-intersecting finite subspaces $I_1^{\al\bt},I_2^{\al\bt},\dots,I_{\tau(\al)}^{\al\bt}$ such that 
$$
p_{\al\bt}(g_i^\al)=\begin{cases} \sum\limits_{j\in I_i^{\al\bt}} g_j^\bt, & \text{if} \;\; I_i^{\al\bt}\neq \vnth, \\
0, & \text{if} \;\; I_i^{\al\bt}= \vnth,
\end{cases}
$$
for $i=1,2,\dots,\tau(\al),\dots\,~.$
\end{itemize}
Then the limit of the direct system $\{G_\al,p_{\al\bt}\}_{\al\in \Lb}$ is a free group.
\end{theorem}

\begin{proof}
Denote by $\ol{B}_\al$ the set of all finite subspaces $\al_t$ of a base $B_\al$ of a group $G_\al$ and by $G_{\al_t}$ a subgroup of group $G_\al$, generated by all elements $\al_t\in \ol{B}_\al$. It is possible to prove that such a group $G_\al$ is the direct limit group of the direct system of subgroups $G_{\al_t}$.

Let $\ol{\Lb}$ be a set $\{(\al,t) \mid \al_t\in \ol{B}_\al\}$. It is considered that $(\al',t')>(\al,t)$, if $\al'>\al$ and $p_{\al\al'}G_{\al_t}\sbs G_{\al'_{t'}}$. It is clear that $\ol{\Lb}$ is a directed set and if we take $G_{\al,t}=G_{\al_t}$ for each pair $(\al,t)\in \ol{\Lb}$, we obtain a direct system $\{G_{(\al,t)},p_{\al\bt}\}$ which satisfies the condition of Theorem~3 \cite{6}. Therefore, the direct limit of the given system is a free abelian group.

Let $G_\infty=\rlm G_\al$ and $G_\infty^*=\rlm G_{(\al,t)}$. Define a homomorphism $\vf:G_\infty^* \to G_\infty$. Since  $(\al',t')>(\al,t)$ we have the following commutative diagram
$$
\xymatrix{
G_{\al,t} \ar[r]^-{\rho_{\al,t}} \ar[dd]_-{p_{\al\al'}} & G_\al \ar[rd]^-{p_\al} \ar[dd]^-{p_{\al\al'}} \\
&& G_\infty \\
G_{\al',t'} \ar[r]_-{\rho_{\al',t'}} & G_{\al'} \ar[ru]_-{p_{\al'}}
}
$$
A homomorphism  $\vf$ is induced by $\vf_{\al,t}=p_\al\rho_{\al,t}$.

$a$) $\vf$ is an epimorphism. Let $x\in G_\infty$ and $x_\al\in G_\al$ be their representatives. Since $G_\al=\rlm G_{\al,t}$, there is a representative $x_{\al,t}\in G_{\al,t}$ of an element $x_\al$. It is clear that a class $x^*\in G_\infty^*$, a representative of which is  $x_{\al,t}$, satisfies the properties $\vf(x^*)=x$.

$b$) $\vf$ is a monomorphism. Let $\vf(x^*)=0$. Since $p_\al \rho_{\al,t}(x_{\al,t})=0$, where $x_{\al,t}$ is a representative of an element $x^*$, there is such $\bt>\al$, that $p_{\al\bt}(\rho_{\al,t}(x_{\al,t}))=0$. Let $G_{\bt,t'}$ be the subgroup of a group $G_\bt$, which is generated by all  $g_j^\bt\in B_\bt$ such that $p_{\al\bt}(g_i^\al)=\sum g_j^\bt$, when $g_i^\al$ runs through the base  $G_{\al,t}$. Since $\rho_{\bt',t'} p_{\al\bt}(x_{\al,t})=p_{\al\bt} \rho_{\al,t}(x_{\al,t})=0$ and $\rho_{\bt,t'}$ are monomorphisms, $p_{\al\bt}(x_{\al,t})=0$ and so $x^*=0$.
\end{proof}

\begin{remark}
Theorem \ref{t2} is a generalization of Theorem 3 \cite{6}, which is proved in the case when the base $B_\al$ is a finite .
\end{remark}

\begin{theorem}
Let $X$ be a locally compact, paracompact Hausdorff space, then  there exists the universal coefficient formula for the {\color{black} Kolmogoroff} homology group: 
$$
0 \lra \Ext(\check{H}^{n+1}(\dot{X},*),G) \lra \os{{\,}_K}{H}_n(X,G) \lra \Hom(\check{H}^n(\dot{X},*),G)\lra 0.
$$
\end{theorem}

\begin{proof}
It is easy to see that the direct system $\{C_f^n(N_\al),\pi_*^{\al\bt}\}$ of groups $C_f^n(N_\al)$ of cochains with integer coefficient group of nerves $N_\al$, where $S_\al\in S$, satisfies the condition of Theorem  \ref{t2}. Therefore, the direct limit  $C_f^n(X) =\rlm(C_f^n(N_\al),\pi_*^{\al\bt})$ is a free group. By {\color{black}Theorem 4.1 \citep[\S III.4]{8},} for the homology group $H_n(\Hom(C_f^*(X),G))$ there exists the Universal Coefficient Formula:

\begin{equation}\label{eq8}
  0 \lra \Ext(H_f^{n+1}(X),G) \lra H_n(\Hom(C_f^*(X),G)) \lra \Hom(H_f^n(X),G) \lra 0.
\end{equation}
Since $\Hom(C_f^*(X),G) \approx C_*^{inf}(X,G)$, by an isomorphism \eqref{eq7} and Theorem \ref{t1}, there exists an isomorphism 
\begin{equation}\label{eq9}
  H_*(\Hom(C_f^*(X),G)) \approx \os{ch}H_*(X,G) \approx \os{{\,}_K}{H}_*(X,G).
\end{equation}
On the other hand, by Theorem 2.1.1 \cite{3} and Theorem 6.9 \cite[\S X.6]{4} we obtain isomorphisms 
\begin{equation}\label{eq10}
  H_f^*(X,G) \approx {H}_{\triangle}^*(X,G) \approx \check{H}^*(\dot{X},*,G),
\end{equation}
where ${H}_{\triangle}^*$ is the Alexandroff homology with proper subcomplexes. Using the exact sequence \eqref{eq8}, by the isomorphisms \eqref{eq9} and \eqref{eq10} we obtain the required statement.  
\end{proof}

\begin{corollary}\label{c3}
An inclusion $\rho:X\to (\dot{X},*)$, where $\dot{X}$ is the one-point Alexandroff compactification of  locally compact, paracompact Hausdorff space $X$, induces an isomorphism 
\begin{equation}\label{eq11}
  \os{{\,}_K}{H}_*(X,G) \approx \os{{\,}_K}{H}_*(\dot{X},*,G).
\end{equation}
\end{corollary}

\begin{corollary}\label{c4}
Since the Kolmogoroff and the Massey homology theories satisfy the condition of uniqueness, in particular the Universal Coefficient Formula \cite[Theorem 4.4]{1},{\color{black} \cite[Theorem 1.5]{BM}}, they are isomorphisms  on the category of compact spaces.
\end{corollary}

\begin{corollary}\label{c5}
By the corollaries $\ref{c2}$ and $\ref{c3}$, the Kolmogoroff and the Massey homologies of locally compact, paracompact Hausdorff spaces are isomorphic the Kolmogoroff and the Massey homologies of compact space, which is the one-point Alexadroff compactification of the given space. Therefore, by corollary  $\ref{c4}$ there is an isomorphism 
$$
\os{{\,}_K}{H}_*(X,G) \approx \os{{\,}_M}{H}_*(X,G)
$$
on the category of locally compact, paracompact {\color{black}Hausdorff} spaces and proper maps.
\end{corollary}

\begin{corollary}\label{c6}
		{\rm a)} If $X$ is a locally compact, paracompact Hausdorff space, then for the system $\{N\}$ of closed neighborhoods $N$ of a closed subspace $A$ of  $X$, there is an infinite exact sequence
	\begin{gather*}
	\cdots \lra \llm^{(2k+1)} \os{{\,}_K}{H}_{n+k+1}(N) \lra \cdots \lra \llm^{(3)} \os{{\,}_K}{H}_{n+2}(N) \lra \llm^{(1)} \os{{\,}_K}{H}_{n+1}(N) 
	\lra \os{{\,}_K}{H}_{n}(A,G) \nonumber \lra\\
	\os{i_n}{\lra} \llm\os{{\,}_K}{H}_{n}(N) \lra \llm^{(2)} \os{{\,}_K}{H}_{n+1}(N) \lra \cdots \lra \llm^{(2k)} \os{{\,}_K}{H}_{n+k}(N) \lra \cdots\,,
	\end{gather*}
	where $\os{{\,}_K}{H}_*(N)=\os{{\,}_K}{H}_*(N,G)$ is the Kolmogoroff homology. 
	
	{\rm b)} If  $X$ is a compact Hausdorff space, then for the system $\{N\}$ of closed neighborhoods $N$ of a closed subspace $A$ of $X$, there is an infinite exact sequence
	\begin{gather*}
	\cdots \lra \llm^{(2k+1)} \os{{\,}_{Mi}}{H}_{n+k+1}(N) \lra \cdots \lra \llm^{(1)} \os{{\,}_{Mi}}{H}_{n+1}(N) \lra \os{{\,}_{Mi}}{H}_{n}(A) \\
	\os{i_n}{\lra} \llm\os{{\,}_{Mi}}{H}_{n}(N) \lra \llm^{(2)} \os{{\,}_{Mi}}{H}_{n+1}(N) \lra \cdots \lra \llm^{(2k)} \os{{\,}_{Mi}}{H}_{n+k}(N) \lra \cdots\,,
	\end{gather*}
	where $\os{{\,}_{Mi}}{H}_*(N)=\os{{\,}_{Mi}}{H}_*(N,G)$ is the Milnor homology \cite{10}.
\end{corollary}

As it is known \cite{6n}, for each countable inverse system  $\{A_k\}$ of abelian groups $A_k$ there is  $\llm^{(i)} \{A_k\}=0$ for $i\geq 2$. By this fact and Theorem \ref{thm5} and Corollary \ref{c6}, we have

\begin{corollary}\label{c7}
	{\rm a)} If $X$ is a locally compact {\color{black}Hausdorff} space with second countable axiom, then for each countable system $\{N_i\}$ of closed neighborhoods of a closed subspace  $A$ of  $X$ there is a short exact sequence 
	$$
	0 \lra \llm^{(1)} \os{{\,}_M}{H}_{n+1} (N_i) \lra \os{{\,}_M}{H}_n(A,G) \lra \llm \os{{\,}_M}{H}_n(N_i) \lra 0,
	$$
	where $\os{{\,}_M}{H}_*$ is the Massey homology \cite{9}.
	
	{\rm b)} If $X$ is a locally compact, paracompact Hausdorff space with second countable axiom, then for each countable system $\{N_i\}$ of closed neighborhoods of a closed subspace $A$ of $X$ there is a short exact sequence 
	$$
	0 \lra \llm^{(1)} \os{{\,}_K}{H}_{n+1} (N_i) \lra \os{{\,}_K}{H}_n(A,G) \lra \llm \os{{\,}_K}{H}_n(N_i) \lra 0,
	$$
	where $\os{{\,}_K}{H}_*$ is the Kolmogoroff homology \cite{7}.
	
	{\rm c)} If $X$ is a compact Hausdorff space with second countable axiom, then for each countable system $\{N_i\}$ of closed neighborhoods of a closed subspace $A$ of $X$ there is a short exact sequence
	$$
	0 \lra \llm^{(1)} \os{{\,}_{Mi}}{H}_{n+1} (N_i) \lra \os{{\,}_{Mi}}{H}_n(A,G) \lra \llm \os{{\,}_{Mi}}{H}_n(N_i) \lra 0,
	$$
	where $\os{{\,}_{Mi}}{H}_*$ is the Milnor homology \cite{10}.
	
	{\rm d)} If $X$ is a compact metric space, then for each countable system $\{N_i\}$ of a closed neighborhoods  of closed subspace $A$ of $X$ there is a short exact sequence
	$$
	0 \lra \llm^{(1)} \os{st}{H}_{n+1} (N_i) \lra \os{st}{H}_n(A,G) \lra \llm \os{st}{H}_n(N_i) \lra 0,
	$$
	where $\os{st}{H}_*$ is the Steenrod  homology \cite{12}.
\end{corollary}

\bibliographystyle{elsarticle-num}

\end{document}